\documentclass[12pt]{article}

\usepackage{amsfonts}

\usepackage{graphicx}
\usepackage{amsmath}
\usepackage{amsthm}
\usepackage{amssymb}
\usepackage{lmodern}
\usepackage{amsmath}
\usepackage{amsthm}
\usepackage{amscd}
\usepackage{amssymb}
\usepackage{graphicx}
\usepackage{textcomp}
\usepackage{color}

\usepackage{mathtools}

\usepackage[latin9]{inputenc}   

\catcode`\=\active\def{\"a}        
\catcode`\=\active\def{\`e}        
\catcode`\=\active\def{\"{\i}}     
\catcode`\=\active\def{\"o}        
\catcode`\=\active\def{\"u}        
\catcode`\Ø=\active\defØ{\"y}        
\catcode`\=\active\def{\"A}        
\catcode`\=\active\def{\"O}        
\catcode`\=\active\def{\"U}        
\catcode`\=\active\def{\'a}        
\catcode`\=\active\def{\'e}        
\catcode`\=\active\def{\'o}        
\catcode`\=\active\def{\'u}        
\catcode`\=\active\def{\'E}        
\catcode`\=\active\def{\`a}        
\catcode`\=\active\def{\`{\i}}     
\catcode`\=\active\def{\`o}        
\catcode`\=\active\def{\`u}        
\catcode`\Ë=\active\defË{\`A}        
\catcode`\=\active\def{\~a}        
\catcode`\=\active\def{\~n}        
\catcode`\=\active\def{\~o}        
\catcode`\Ì=\active\defÌ{\~A}        
\catcode`\=\active\def{\~N}        
\catcode`\Í=\active\defÍ{\~O}        
\catcode`\=\active\def{\^a}        
\catcode`\=\active\def{\^e}        
\catcode`\=\active\def{\^i}     
\catcode`\=\active\def{\^o}        
\catcode`\=\active\def{\^u}        
\newcommand{\be}{\begin{equation}}
\newcommand{\ee}{\end{equation}}

\newcommand{\ba}{\begin{eqnarray*}}
\newcommand{\ea}{\end{eqnarray*}}

\theoremstyle{plain}
\newtheorem{theorem}{Theorem}

\theoremstyle{definition}
\newtheorem{definition}{Definition}

\newtheorem{lemma}{Lemma}
\newtheorem{corollary}{Corollary}
\newtheorem{remark}{Remark}
\newtheorem{example}{Example}

\begin{document}

\begin{center}
 \textbf{New expansion of the Riemann zeta function}
  \medskip
 
 B. Candelpergher
 
  \medskip
 \textsl{\small Universit\'e Côte d'Azur, CNRS, LJAD (UMR 7351), Nice, France.   candel@univ-cotedazur.fr}

 \end{center}
  \bigskip
  
\begin{abstract}  This article presents polynomial expansions for the Dirichlet eta function and Riemann zeta function that are convergent  in the critical strip. To do this  we introduce a family of hypergeometric  polynomials,  whose roots lie on the line $\{\Re(s)=1/2\}$, and that  are related to Meixner-Pollaczek  polynomials. 
We also obtain orthonormal expansions for  eta and zeta  restricted to 
 the line $\{\Re(s)=1/2\}$. The coefficients of these  expansions are given explicitly as  linear combinations with rational coefficients of $\log(2),$  Euler's constant $\gamma$,  and  zeta values at positive integers.

\end{abstract}

\section{Introduction}

Some polynomial expansions of  Riemann's zeta function  (cf.[E-K-K], [F-V],[B],[M], [V])
 are based on the Newton interpolation formula for functions related to the zeta function. For example
 $$
\zeta(s)-\frac1{s-1}=\sum_{m= 0}^{+\infty}(-1)^mc_m \binom{s-1}{m},
$$
with 
$$c_m=\gamma+\sum_{k= 1}^m\binom{m}{k}(-1)^k\big(\zeta(k+1)-\frac1{k}\big).$$
Others (cf.[E-G])  
are linked to the well known  expansion
$$\zeta(s)-\frac1{s-1}=\sum_{m= 0}^{+\infty}(-1)^m\gamma_m\frac{(s-1)^m}{m!}
$$
where $\gamma_m$ are the Stieltjes constants. 

\medskip

The polynomial expansions that we give in this paper  are related to the Mellin transform of the functions  $x\mapsto  e^{- x}L_m(2x) $ where
  $(L_m)_{m\geq 0}$ is the sequence of Laguerre polynomials   defined by
\begin{equation}\label{I1}
L_m(x)=\sum_{k= 0}^m\binom{m}{k}(-1)^k\frac{x^k}{k!}.
\end{equation}
Let $P_m$ the polynomial defined, for $\Re(s)>0$, by
\begin{equation}\label{I3}
P_m(s)
=\frac 2{\Gamma(s)}\int_0^{+\infty} e^{- x}L_m(2x)x^{s-1}\ dx.
\end{equation}
We prove (cf.Thm. 1) that the sequence of polynomials 
$
Q_m:t\mapsto \frac12P_m(\frac12+it)
$
is an orthonormal basis of the space $L^2\big(\mathbb R,\frac{dt}{\cosh(\pi t)}\big)$. 
This implies  that the roots of polynomials $P_m$  are on the line $\{\Re(s)=1/2\}.$

\noindent The polynomials $P_m$ are also given by the generating function
\begin{equation}\label{I2}
z^{-s}=\sum_{m= 0}^{+\infty}\frac1{1+z}\big(\frac{z-1}{z+1}\big)^mP_m(s)\ \text{ for }\ \Re(z)>0.
\end{equation}
If we assign the integer values 1, 2, 3, ... to $z$ in this formula,  then for any   sequence $(c_n)_{n\geq1}$ of complex numbers, we get a family  of relations 
\begin{equation}\label{I4}
\frac{c_n}{n^{s}}=\sum_{m= 0}^{+\infty}\frac{c_n}{1+n}\big(\frac{n-1}{n+1}\big)^mP_m(s), \  \text{ for }\  n=1,2,3,\dots
\end{equation}
It's tempting to sum up these relations  to obtain at least formally
$$
\sum_{n\geq 1}\frac {c_n}{n^s}=\sum_{m= 0}^{+\infty}\Big(\sum_{n\geq 1}\frac{c_n}{1+n}\big(\frac{n-1}{n+1}\big)^m\Big)P_m(s).
$$

We begin by examining the case where $c_n=(-1)^{n-1}$ for every $n\geq 1$,  which correspond to  Dirichlet eta function
$$
\eta(s)=\sum_{n=1}^{+\infty}\frac{(-1)^{n-1}}{n^s}.
$$
In this case the series  $\sum_{n\geq 1}\frac{(-1)^{n-1}}{1+n}\big(\frac{n-1}{n+1}\big)^m$ are convergent and, for $0<\Re(s)<1$, we obtain (cf. Thm. 2) the convergent expansion 
$$
\eta(s) =\sum_{m= 0}^{+\infty}a_mP_m(s)\  \text{ where }\  a_m=\sum_{n= 1}^{+\infty}\frac{(-1)^{n-1}}{1+n}\big(\frac{n-1}{n+1}\big)^m.
$$

We then consider the case where  $c_n=1$ for every $n\geq 1$, which correspond to  Riemann zeta function
$$
\zeta(s)=\sum_{n=1}^{+\infty}\frac{1}{n^s}.
$$
 In this case the series $\sum_{n\geq 1}\frac{1}{1+n}\big(\frac{n-1}{n+1}\big)^m$ are divergent.
 To make sense of these divergent sums, we will use Ramanujan's summation method, which we will describe in  the appendix.
If we denote by $\sum_{n\geq 1}^{\mathcal{R}}$ the sum of a series in the Ramanujan sense, we obtain  
 (cf. Thm. 4), for $0<\Re(s)<1$, the convergent expansion 
 $$
\sum_{n\geq 1}^{\mathcal{R}}\frac1{n^s}=\sum_{m= 0}^{+\infty}z_mP_m(s) \ \text { where }\ 
z_m=\sum_{n\geq 1}^{\mathcal{R}}\frac1{n+1}\Big(\frac{n-1}{n+1}\Big)^m.
 $$
 Since  we will see in the appendix, that  for $s\neq 1$ we have
$$
\sum_{n\geq 1}^{\mathcal{R}}\frac1{n^s}=\zeta(s)-\frac1{s-1},
$$
then, after an explicit evaluation of $z_m$,  we obtain (cf. Thm. 6), for $ 0<\Re(s)<1$, the following convergent expansion
\begin{equation*}
\zeta(s)=\sum_{m= 0}^{+\infty}x_mP_m(s),
\ \text{ with }\ 
x_m=\gamma+\sum_{k=1}^m \binom{m}{k}(-2)^k\zeta(k+1) +r_m
\end{equation*}
where $r_m$ is the rational  
 $$
 r_m=(-1)^{m+1}+\sum_{k=1}^m \frac{1+(-1)^{k-1}}{k}.
 $$

\section{ The sequence of polynomials $(P_m)$}

\subsection{Definition and properties}
We will begin by defining the sequence of  polynomials $(P_m)$ using a generating function.
\noindent For ${m\geq 0}$, let $\psi_m$  the function defined in the half-plane $\{Re(z)>0\}$ by 
\begin{equation}
\psi_m(z)=\frac1{1+z}\big(\frac{z-1}{z+1}\big)^m,
\end{equation}
and let  for $s\in \mathbb C$ the complex sequence $(P_m(s))$  given by the generating function
\begin{equation}
\frac{2}{1-\xi }\Big({\frac{1-\xi}{1+\xi}}\Big)^{s}=\sum_{m= 0}^{+\infty}P_m(s)\ \xi^m\  \text{ for }\ \ \ \vert\xi\vert<1.
\end{equation}
Then, with $\xi=\frac{z-1}{z+1}$,  we have for $s\in \mathbb C$ and $Re(z)>0$, the expansion 
\begin{equation}
z^{-s}=\sum_{m= 0}^{+\infty}\psi_m(z)\, P_m(s).
\end{equation}
Let   $\xi=\frac{x}{1+x}$, with $x\in]0,1[$, then by (6)  we get
\begin{equation*}
\sum_{m= 0}^{+\infty}P_m(s)\Big(\frac{x}{1+x}\Big)^{m+1}=2x(1+2x)^{-s}=\sum_{k=  0}^{+\infty}(-1)^k2^{k+1}\frac{(s)_k}{k!}x^{k+1},
\end{equation*}
where  the Pochhammer symbol $(s)_k$ is the rising factorial 
$$(s)_k=\frac{\Gamma(s+k)}{\Gamma(s)}=s\cdots(s+k-1) \text{ if }  k\geq 1, \text{ and }  (s)_0=1.
$$
Now the well-known equivalence related to Euler's transformation (cf.[H1]) :
\begin{equation*}
\sum_{m\geq 0}b_{m}(\frac x{1+x})^{m+1}=\sum_{k\geq 0}a_kx^{k+1}
\ \Leftrightarrow \  b_{m}=\sum_{k= 0}^m\binom{m}{k}\ a_{k},
\end{equation*}
 gives for $P_m(s)$ the following  explicit polynomial expression   
\begin{equation}
P_m(s)= \sum_{k= 0}^m\binom{m}{k}(-1)^{k}2^{k+1}\ \frac{(s)_k}{k!}.
\end{equation}
Note that by inverting this  binomial transformation we have also 
\begin{equation}
\frac12 \sum_{k= 0}^m\binom{m}{k}(-1)^{k}P_k(s)=2^{m}\ \frac{(s)_m}{m!}.
\end{equation}
It  is easily verified with (8) and  the definition of the Laguerre polynomials  that $P_m(s)$ is related,  for $ \Re(s)>0,$  to the Laguerre polynomial $L_m(x)$ by  
\begin{equation}
{\Gamma(s)}P_m(s)
=2\int_0^{+\infty} e^{- x}L_m(2x)x^{s-1}\ dx.
\end{equation}
For $0<\Re(s)<1$,  we have another  integral expression for $P_m(s) $ which  is obtained by writing
$$
\frac{(s)_k}{k!}=\frac1{\Gamma(s)\Gamma(1-s)}\int_0^1t^{s-1+k}(1-t)^{-s} \  dt\  \text{ for }\ k=0,1,2,3,\dots
$$
thus, using (8), we get    for  $ 0<\Re(s)<1$
\begin{equation}
P_m(s)=2\frac{\sin(\pi s)}{\pi}\int_0^1t^{s-1}(1-t)^{-s} (1-2t)^m\  dt.
 \end{equation}
This integral expression allows us to prove the following lemma which will be useful   later.

\begin{lemma}
\it  For  $0<\Re(s)<1$, we have, for every   $m\geq 1$

\begin{equation*}
\vert P_m(s)\vert \leq e^{\pi\vert Im (s)\vert}\Big(\frac{\Gamma(\Re(s))}{m^{\Re(s)}}+\frac{\Gamma(1-\Re(s))}{m^{1-\Re(s)}}\Big).
 \end{equation*}

\end{lemma}

\begin{proof}

Using (11) we get 
\begin{eqnarray*}
\vert P_m(s)\vert &\leq&2\frac{\vert\sin(\pi s)\vert}{\pi}\int_0^1t^{\Re(s)-1}(1-t)^{-\Re(s)} \vert1-2t\vert^m \ dt,\\
&\leq&2\frac{\vert\sin(\pi s)\vert}{\pi}\int_{-1}^1(1-u)^{\Re(s)-1}(1+u)^{-\Re(s)} \vert u\vert^m \ du,\\
&\leq&e^{\pi\vert Im (s)\vert}\int_{-1}^1(1-u)^{\Re(s)-1}(1+u)^{-\Re(s)} \vert u\vert^m \ du.
\end{eqnarray*}
Let $a=\Re(s)$, this last  integral is :
$$
I_a=\int_{0}^1(1-u)^{a-1}(1+u)^{-a} u^m \ du+\int_{-1}^0(1-u)^{a-1}(1+u)^{-a}  (-u)^m \ du.
$$
Since $0<a<1$,  we have  
\begin{eqnarray*}
I_a\leq\int_{0}^1(1-u)^{a-1} u^m du+\int_{-1}^0(1+u)^{-a}  (-u)^m\  du
&=&\int_{0}^1(1-u)^{a-1} u^m du+\int_{0}^1(1-u)^{-a}  u^m\  du.
\end{eqnarray*}
Now if $0<\alpha<1$ we have
$$
\int_{0}^1{(1-u)^{\alpha-1} }u^m \ du=\int_{0}^{+\infty}{e^{-t}}{(1-e^{-t})^{\alpha-1}} e^{-mt} \ dt \leq \int_{0}^{+\infty}t^{\alpha-1} e^{-mt} \ dt=\Gamma(\alpha)\ m^{-\alpha},
$$
(the last inequality is justified by  $e^{-t}\leq\frac{1-e^{-t}}t\leq{1}$, for every $t>0$). 
With $\alpha=a $ and $\alpha=1-a $, we obtain 
$$I_a\leq \frac{\Gamma(a)}{m^a}+\frac{\Gamma(1-a)}{m^{1-a}}.$$

\end{proof}

\begin{remark}
The polynomials $P_m$ are related with  Gauss hypergeometric function
$$
F(a,b,c,z)=\sum_{k\geq 0}\frac{(a)_k(b)_k}{(c)_k}\frac{z^k}{k!},
$$
since it  is easily verified with (8) that, for $m\geq 1$, we have 
$$
P_m(s)=2\ F(-m,s,1,2).
$$
By classical properties of hypergeometric function (cf. [G-R]), we  get the relations 
\begin{equation}
P_m(1-s)=(-1)^mP_m(s),
\end{equation}
\begin{equation}
sP_m(s+1)-(s-1)P_m(s-1)=(2m+1)P_m(s).
\end{equation}

For the first values of $m$ we have 
\begin{eqnarray*}
P_0(s)&=&2,\\
P_1(s)&=&2-4s,\\
P_2(s)&=&2-4s+4s^2,\\
P_3(s)&=&2-\frac{16}3 s+4s^2-\frac83 s^3,\\
P_4(s)&=&2-\frac{16}3 s+\frac{20}3s^2-\frac83 s^3+\frac43s^4.\\
\end{eqnarray*}

\end{remark}

\subsection{Orthogonality relations}

  The polynomials $P_m$ are related by Mellin transform to  the  functions $\varphi_m$ defined by
\begin{equation}
\varphi_m(x)= \sqrt2\ e^{- x}L_m(2x).
\end{equation}
Indeed we verify  with (10)  that, for  $\Re(s)>0$,   we have 
\begin{equation}
\mathcal{M}\varphi_m(s)=
\int_0^{+\infty} \varphi_m(x)\ x^{s-1}\ dx= \frac1{\sqrt2}{\Gamma(s)}P_m(s).
\end{equation}
Since the sequence $(\varphi_m)_{m\geq 0}$  is (cf.[Leb]) an orthonormal basis of the space $L^2(]0,+\infty[,\ dx)$, we get the following result.

\begin{theorem}
The sequence $(Q_m)_{m\geq 0}$ of polynomials 
\begin{equation}
Q_m:t\mapsto \frac12P_m(\frac12+it)
\end{equation}
  is an orthonormal basis  of the space $L^2\big(\mathbb R,\frac{dt}{\cosh(\pi t)}\big)$.

  \end{theorem}
  
  \medskip
  
  \begin{proof}

By the 
 Mellin-Parseval  formula (cf. [T]), we have  
$$
\frac{1}{2\pi }\int_{-\infty }^{+\infty }%
\overline{\mathcal{M}(\varphi_m)(\frac{1}2+it)}\ \mathcal{M}(\varphi_n)(\frac{1}2+it)\ dt=\int_{0}^{+\infty }\overline{\varphi_m(x)}\varphi_n(x)\ dx=\delta_{(m,n)},
$$
with $\delta$ the usual Kronecker symbol. 
That is
$$
\frac{1}{4\pi }\int_{-\infty }^{+\infty }\overline{\Gamma(\frac{1}2+it)}\overline{P_m(\frac12+it)}\ \Gamma(\frac{1}2+it)
P_n(\frac12+it)\ dt=\delta_{(m,n)}.
$$
Since
$$\Gamma(\frac{1}2+it)
\overline{\Gamma(\frac{1}2+it)}
 =\frac{\pi}{\cosh\pi t },
$$
 we have the orthogonality relation
$$
\frac{1}{4}\int_{-\infty }^{+\infty }\overline{P_m(\frac12+it)}
\ P_n(\frac12+it)\frac{dt}{\cosh\pi t }=\delta_{(m,n)}.
$$
Thus the   sequence $(Q_m)_{m\geq 0}$ is an orthonormal system in $L^2\big(\mathbb R,\frac{dt}{\cosh(\pi t)}\big)$.
It is also an  orthonormal basis of this space since  a classical condition for the density of the polynomials in  $L^2(\mathbb R,\rho(t)  dt)$ is 
existence of $a>0$ such that
  $$
  \int_{-\infty }^{+\infty }e^{a\vert t\vert}\rho(t)dt <+\infty.   $$
  
   \end{proof}

\medskip
\begin{corollary}

 For every $m\geq 1$, the roots  of 
 $P_m$  are on the line $\{\Re(s)=\frac12\}$.
\end{corollary}

\begin{proof}
By a well-known result  on orthogonal polynomials we deduce from the preceding theorem  that, for every $m\geq 1$, the roots of   $Q_m$  are real. 
 Since $P_m(\frac12+it)=2Q_m(t)$  then the roots of 
 $P_m$  are located on the line $\{\Re(s)=\frac12\}$. 
 \end{proof}
 
 \medskip
 
\begin{remark}
By (12)  and (13) we have, 
for every $m\geq 0$, the relations
$$
\overline{Q_m(t)}=Q_m(-t)=(-1)^mQ_m(t),
$$
$$
TQ_m=(2m+1)Q_m,
$$
where $T$ is the operator on $\mathbb R[t]$ defined by  
$$Tf(t)=
\big(\frac12+it\big)f(t-i)
+\big(\frac12-it\big)f(t+i).
$$
Since $Q_m(t)=F(-m,\frac12+it,1,2)$, we have  also the three terms relation
$$
(m+1)Q_{m+1}(t)=2itQ_m(t)+mQ_{m-1}(t).
$$
These polynomials are related to the Meixner-Pollaczek polynomials (cf. [K-L-S], [I-S]) defined by
$$
\mathcal P^{(1/2)}_m(t,\frac{\pi}2)=(i^m)F(-m,\frac12+it,1,2).
$$
We have for every integer $m\geq 0$
$$Q_m(t)=(-i)^m\mathcal P^{(1/2)}_m(t,\frac{\pi}2).
$$
For the first values of $m$ we get 
\begin{eqnarray*}
&&Q_{{0}}\, = \,1\\
&&Q_{{1}}(t)\, = \,-2\,it\\
&&Q_{{2}}(t)\, = \,\frac12-2\,{t}^{2}\\
&&Q_{{3}}(t)\, = \,-\frac53\,it+\frac43\,i{t}^{3}\\
&&Q_{{4}}(t)\, = \,\frac38-\frac73\,{t}^{2}+\frac23\,{t}^{4}\\
&&\dots
\end{eqnarray*}

\end{remark}

\section{An expansion of the Dirichlet eta function}

\subsection{A convergent expansion}
Let the \textit{Dirichlet eta function} defined, for $\Re(s)>0$, by the \text{Euler alternating series}
$$
\eta(s)=\sum_{n=1}^{+\infty}\frac{(-1)^{n-1}}{n^s}=(1-2^{1-s})\zeta(s).
$$
As we described  in the introduction, the relation (7) with $z=1,2,3,\dots $, gives by summation 
 \begin{equation}
\eta(s)
=\sum_{n=1}^{+\infty}\sum_{m= 0}^{+\infty}\frac{(-1)^{n+1}}{n+1}\Big(\frac{n-1}{n+1}\Big)^mP_m(s).
\end{equation}
The next lemma show that we can swap these two summations.

\medskip

\begin{lemma} For  $0<\Re(s)<1$ we have 
$$
\eta(s)=\sum_{m= 0}^{+\infty}a_mP_m(s) \ \text { with }\ 
 a_m=\sum_{n= 1}^{+\infty}\frac{(-1)^{n+1}}{n+1}\Big(\frac{n-1}{n+1}\Big)^m.
 $$
 \end{lemma}
 
 \begin{proof}
 
 \medskip
 
By residue theorem  we have (cf.[Lin])
 $$
 \sum_{n=1}^{+\infty}\frac{(-1)^{n-1}}{n^s}=\frac1{2i\pi}\int_{\frac12-i\infty}^{\frac12+i\infty}\frac1{z^s}\frac{\pi}{\sin(\pi z)}dz
 =\frac1{2}\int_{-\infty}^{+\infty}\frac{1}{(\frac12+it)^s}\frac1{\cosh(\pi t)}dt.
$$
Since by (7) we can write 
$$
\frac{1}{(\frac12+it)^s}
=\frac1{\frac32+it}\sum_{m= 0}^{+\infty}\Big(\frac{-\frac12+it}{\frac32+it}\Big)^mP_m(s),$$
then we get  
$$
\sum_{n=1}^{+\infty}\frac{(-1)^{n-1}}{n^s}=\frac1{2}\int_{-\infty}^{+\infty}\frac1{\frac32+it}\sum_{m= 0}^{+\infty}\Big(\frac{-\frac12+it}{\frac32+it}\Big)^mP_m(s)\frac1{\cosh(\pi t)}dt.
$$
\bigskip

\noindent Now we can swap   the  signs $\sum$ and $\int$ since, using  lemma 1, 
we have
 $$
\Big\vert\frac1{\frac32+it}\Big(\frac{-\frac12+it}{\frac32+it}\Big)^mP_m(s)\Big\vert\leq\frac1{\sqrt{t^2+9/4}}\Big(\frac {\sqrt{t^2+1/4}}{\sqrt{t^2+9/4}}\Big)^me^{\pi\vert Im (s)}\Big(\frac{\Gamma(\Re(s))}{m^{\Re(s)}}+\frac{\Gamma(1-\Re(s))}{m^{1-\Re(s)}}\Big) 
$$
and for $0<a<1$ we have for every $t\in \mathbb R$
$$
\sum_{m= 0}^{+\infty}\frac1{\sqrt{t^2+9/4}}\Big(\frac {\sqrt{t^2+1/4}}{\sqrt{t^2+9/4}}\Big)^m\frac1{m^a}\leq \sum_{m= 0}^{+\infty}\frac1{\sqrt{t^2+9/4}}\Big(\frac {\sqrt{t^2+1/4}}{\sqrt{t^2+9/4}}\Big)^m\leq t.
$$
Thus we obtain
$$
\sum_{n=1}^{+\infty}\frac{(-1)^{n-1}}{n^s}=\sum_{m= 0}^{+\infty}\frac1{2}\int_{-\infty}^{+\infty}\frac1{\frac32+it}\Big(\frac{-\frac12+it}{\frac32+it}\Big)^m\frac1{\cosh(\pi t)}dt\  P_m(s).
$$
We can write the  integrals in this sum as
$$
 \frac1{2i\pi}\int_{\frac12-i\infty}^{\frac12+i\infty}\frac1{z+1}\Big(\frac{z-1}{z+1}\Big)^m\frac{\pi}{\sin(\pi z)}dz
$$
which, again by the residue theorem, are equal to
$$
\sum_{n=1}^{+\infty}\frac{(-1)^{n+1}}{n+1}\Big(\frac{n-1}{n+1}\Big)^m.
$$

 \end{proof}

 \begin{theorem}

For $ 0<\Re(s)<1$,  we have
\begin{equation*}
\eta(s)=\sum_{m= 0}^{+\infty}a_mP_m(s)
 \end{equation*}
with 
  $a_0=1-\log(2)$ and   
   $$
   a_m=(-1)^m-\log(2)+\sum_{k=1}^m  \binom{m}{k}(-1)^k(1-2^k)\zeta(k+1)\ \text{ for }\ m\geq 1.
   $$

    \end{theorem}

 \begin{proof} We  observe that the coefficients $a_m$ given by  the preceding lemma are also 
$$
a_m=\sum_{n= 2}^{+\infty}\frac{(-1)^{n}}{n}\Big(1-\frac{2}{n}\Big)^m,
$$
and by  binomial expansion  we get for $m\geq 0$
\begin{eqnarray}
 a_m
 =\sum_{k=0}^m  \binom{m}{k}(-2)^k\sum_{n=2}^{+\infty}\frac{(-1)^{n}}{n^{k+1}}=\sum_{k=0}^m  \binom{m}{k}(-1)^k2^k(1-\eta(k+1)).
\end{eqnarray}
Thus $a_0=\sum_{n=2}^{+\infty}\frac{(-1)^{n}}{n}=1-\log(2)$ and for $m\geq 1$
\begin{eqnarray*}
a_m&=&
1-\log(2)+\sum_{k=1}^m  \binom{m}{k}(-1)^k2^k(1-(1-2^{-k})\zeta(k+1))\\
&=&1-\log(2)+\sum_{k=1}^m  \binom{m}{k}(-1)^k2^k+\sum_{k=1}^m  \binom{m}{k}(-1)^k(1-2^{k})\zeta(k+1).
\end{eqnarray*}
Thus the expansion of theorem 2 is simply a consequence of Lemma 2 .

 \end{proof}
   
   \bigskip

For the first values of  $m$ we get 
\begin{eqnarray*}
a_0&=&1-\log(2),\\
a_1&=&-1-\log(2)+\zeta(2),\\
a_2&=&1-\log(2)+2\zeta(2)-3\zeta(3),\\
a_3&=& -1-\log(2)+3\zeta(2)-9\zeta(3)+7\zeta(4),\\
a_4&=& 1-\log(2)+4\zeta(2)-18\zeta(3)+28\zeta(4)-15\zeta(5).\\
 \end{eqnarray*}

   \bigskip
   
   \subsection{An orthonormal expansion}
   Let the function $g_0$ defined, for $t>0$, by 
$$
g_0(t)=\frac{e^{-t}}{e^{-t}+1}.
$$
It is well-known that,  for $\Re(s)>0 $, the Mellin transform of $g_0$ is  
$$\mathcal{M}(g_0)(s)=\Gamma(s)\eta(s).
$$
Since $e^{-t}g_0(t)=e^{-t}-g_0(t)$, then  we get
$$\mathcal{M}(e^{-t}g_0)(s)=\Gamma(s)(1-\eta(s)) \text{ if }\ \Re(s)>0.$$ 
Thus, for $k=0,1,2,3,\dots,$ we have
$$
\int_0^{+\infty}{e^{-t}}    \frac{t^k}{k!}\  g_0(t)\ dt =1-\eta(k+1),
$$
and  by definition of the Laguerre polynomials, we obtain 
\begin{equation*}
\sum_{k=0}^m \binom{m}{k}(-1)^k2^k
(1-\eta(k+1))= \int_0^{+\infty}e^{-t} L_{m}(2t)g_0(t)\ dt.
$$
By  (14) and (18)  this means that 
$$
a_m=\frac1{\sqrt2}\int_0^{+\infty}\varphi_m(t)g_0(t) \ dt. 
\end{equation*}
Consequently we have, in the space $L^2(]0,+\infty[,dx)$,
the orthonormal expansion
 \begin{equation}
g_0=\sqrt{2}\, \sum_{m=0}^{+\infty} a_m\, \varphi_m.
\end{equation}

We now prove  that the expansion of $\eta$ given in theorem 2 is  an  orthonormal expansion if we restrict it to the line $\{\Re(s)=\frac12\}$.

  \begin{theorem}

For every  integer $m\geq 0$ we have 
\begin{equation}
a_m=\frac1{2}\int_{-\infty }^{+\infty }\overline{Q_m(t)}\ \eta(\frac12+it)\frac{dt}{\cosh(\pi  t)}
\end{equation}

In the space $L^2\big(\mathbb R,\frac{dt}{\cosh(\pi t)}\big)$, the function 
$\eta_{1/2}:t\mapsto\eta(\frac{1}2+it)$ has the orthonormal expansion
$$
\eta_{1/2}=2\sum_{m=0}^{+\infty} a_m\, Q_m.
$$

\end{theorem}

\begin{proof}
 \noindent By  Mellin-Parseval  formula (cf. [T]), we have using (19)
$$
a_m=\frac1{\sqrt2}\int_{0}^{+\infty }{\varphi_m(t)}g_0(t)\ dt=\frac{1}{2\pi\sqrt2 }\int_{-\infty }^{+\infty }
\overline{\mathcal{M}(\varphi_m)(\frac{1}2+it)}\ \mathcal{M}(g_0)(\frac{1}2+it)\ dt,
$$
 thus we get
\begin{eqnarray*}
a_m&=&\frac{1}{2\pi }\int_{-\infty }^{+\infty }
\overline{\Gamma(\frac{1}2+it)\frac12P_m(\frac12+it)}\ \Gamma(\frac{1}2+it)\eta(\frac{1}2+it)\ dt\\
&=&\frac{1}{2}\int_{-\infty }^{+\infty }
\overline{\frac12P_m(\frac12+it)}\ \eta(\frac{1}2+it)\frac{dt}{\cosh(\pi  t)}.
\end{eqnarray*}
 
\noindent This gives  
$$a_m=\frac{1}{2}\int_{-\infty }^{+\infty }
\overline{Q_m(\frac12+it)}\ \eta(\frac{1}2+it)\frac{dt}{\cosh(\pi  t)}.
$$
The second assertion follows from the fact that the function $ \eta_{1/2}$ is in $L^2\big(\mathbb R,\frac{dt}{\cosh(\pi t)}\big)$ since
$$
\int_{-\infty }^{+\infty }
\vert \eta(\frac{1}2+it)\vert ^2\frac{dt}{\cosh(\pi  t)}=
\frac14\int_{-\infty }^{+\infty }
\vert 1-\sqrt2 e^{-it\log(2)}\vert ^2\vert \zeta(\frac{1}2+it)\vert ^2\frac{dt}{\cosh(\pi  t)}<+\infty.
$$

\end{proof}

\begin{remark}
By the preceding theorem and using (19) we get 
$$
\int_{-\infty }^{+\infty }
\vert \eta(\frac{1}2+it)\vert ^2\frac{dt}{\cosh(\pi  t)}=4\sum_{m=0}^{+\infty} a_m^2=2\int_0^{+\infty} \big(\frac{e^{-t}}{e^{-t}+1}\big)^2 dt=2\log(2)-1.
$$ 
\end{remark}

\medskip

\begin{corollary}
For any integer $m\geq 0$ we have
$$1-\eta(m+1)=\frac12\int_{-\infty }^{+\infty }\eta(\frac12+it)\frac{(\frac12-it)_m}{m!}\frac{dt}{\cosh(\pi  t)}.
$$

\end{corollary}

\medskip

\begin{proof}

Let $A(k)=1-\eta(k+1)$ for $k\geq 0$,
 then, by (18), we have 
 $$
 a_m=\sum_{k=0}^m  \binom{m}{k}(-1)^k2^kA(k),
 $$
and by binomial inversion we get
 $$
 2^mA(m)=\sum_{k=0}^m  \binom{m}{k}(-1)^ka_k. $$
Using (20) we obtain the integral expression
\begin{eqnarray*}
A(m)
&=&\frac1{2^{m+1}}\int_{-\infty }^{+\infty }%
\sum_{k=0}^m \binom{m}{k}(-1)^k\overline{Q_k(t)}\ \eta(\frac{1}2+it)\frac{dt}{\cosh(\pi  t)}.
\end{eqnarray*}
Since  by (9) we have
$$
 \sum_{k= 0}^m\binom{m}{k}(-1)^{k}\overline{Q_k(t)}= \sum_{k= 0}^m\binom{m}{k}(-1)^{k}\overline{\frac12P_k(\frac12+it)}
=2^{m}\ \frac{(\frac12-it)_m}{m!},
$$
then we obtain
$$
A(m)=\frac12\int_{-\infty }^{+\infty }\frac{(\frac12-it)_m}{m!}\eta(\frac12+it)\frac{dt}{\cosh(\pi  t)}.
$$
\end{proof}

\section{Expansion of the zeta function}

We will apply the  previous method  to obtain an expansion of the zeta function, but the intermediate series are divergent in this case, which leads us to use 
 Ramanujan's summation (see Appendix).

 \subsection{Expansion of $\sum_{n\geq 1}^{\mathcal{R}}\frac1{n^s}$}

\medskip

\begin{theorem}

For $0<\Re(s)<1,$
we have
 \begin{equation}
\sum_{n\geq 1}^{\mathcal{R}}\frac1{n^s}=\sum_{m= 0}^{+\infty}z_mP_m(s) \ \text { where }\ 
z_m=\sum_{n\geq 1}^{\mathcal{R}}\frac1{n+1}\Big(\frac{n-1}{n+1}\Big)^m.
 \end{equation}
This gives, for $0<\Re(s)<1,$  the convergent expansion
\begin{equation}
\zeta(s)=\frac1{s-1}+\sum_{m= 0}^{+\infty}z_m\, P_m(s).
\end{equation}

\end{theorem}

\medskip

\begin{proof}
If we apply  (7) with $z=1,2,3,\dots,$ and use the Ramanujan summation,
then we get
 \begin{equation}
\sum_{n\geq 1}^{\mathcal{R}}\frac1{n^s}
=\sum_{n\geq 1}^{\mathcal{R}}\sum_{m= 0}^{+\infty}\psi_m(n)P_m(s)
=\sum_{n\geq 1}^{\mathcal{R}}\sum_{m= 0}^{+\infty}\frac1{n+1}\Big(\frac{n-1}{n+1}\Big)^mP_m(s).
\end{equation}
We now prove  that, for  $0<\Re(s)<1$, we can swap these  two summations, that is to say, the series 
\begin{equation}
\sum_{m\geq 0}\sum_{n\geq 1}^{\mathcal{R}}\frac1{n+1}\Big(\frac{n-1}{n+1}\Big)^mP_m(s)
\end{equation}
 is convergent for every $0<\Re(s)<1$, and   its sum  is equal to 
 $$
\sum_{n\geq 1}^{\mathcal{R}}\sum_{m= 0}^{+\infty}\frac1{n+1}\Big(\frac{n-1}{n+1}\Big)^mP_m(s).
$$ 
By definition (R5) of Ramanujan's summation, we have the integral expression 
$$
 \sum_{n\geq 1}^{\mathcal{R}}\frac1{n+1}\Big(\frac{n-1}{n+1}\Big)^m
 =c_m+i\int_0^{+\infty}\Big( \frac1{2+it}\big(\frac{it}{2+it}\big)^m-\frac1{2-it}\big(\frac{-it}{2-it}\big)^m\Big)\frac1{e^{2\pi t}-1}\ dt
 $$
 with $c_0=1/4$ if $m=0$, and $c_m=0$ if $m\geq 1$.
 Since $c_0P_0(s)=1/2$, then the series (24) is also 
 \begin{eqnarray*}
 \frac12+i\sum_{m\geq0}\int_0^{+\infty}\Big( \frac1{2+it}\big(\frac{it}{2+it}\big)^m-\frac1{2-it}\big(\frac{-it}{2-it}\big)^m\Big)\frac1{e^{2\pi t}-1}\ dt\ P_m(s).
 \end{eqnarray*}
 To prove the convergence of this series  and  
swap $\sum $ and $\int$ we use the following lemma.  

 \medskip
 
 \begin{lemma}
 {\it
 For $0<\Re(s)<1$, the series
\begin{equation}
 \sum_{m\geq 0}\Big\vert\Big( \frac1{2+it}\big(\frac{it}{2+it}\big)^m-\frac1{2-it}\big(\frac{-it}{2-it}\big)^m\Big)\ P_m(s)\Big\vert
 \end{equation}
is convergent for every $t>0$, and its sum verifies 
\begin{equation*}
\int_0^{+\infty} \sum_{m= 0}^{+\infty}\Big\vert\Big( \frac1{2+it}\big(\frac{it}{2+it}\big)^m-\frac1{2-it}\big(\frac{-it}{2-it}\big)^m\Big)\ P_m(s)\Big\vert\  \frac1{e^{2\pi t}-1}\ dt<+\infty.
 \end{equation*}
}
 
 \end{lemma}

By this  lemma  the series (24) is convergent for every $0<\Re(s)<1$, and  its sum  is    
$$
\frac12
+i\int_0^{+\infty}\Big( \frac1{2+it}\sum_{m= 0}^{+\infty}\big(\frac{it}{2+it}\big)^mP_m(s)-\frac1{2-it}\sum_{m= 0}^{+\infty}\big(\frac{-it}{2-it}\big)^mP_m(s)\Big)
\frac1{e^{2\pi t}-1}\ dt.
$$
Using (7) with 
$\xi=\frac{1}{1\pm it}$, this last expression is   equal to
$$
\frac12+i\int_0^{+\infty}\Big(\big(\frac{1}{1+it}\big)^s-\big(\frac{1}{1-it}\big)^s\Big)\frac1{e^{2\pi t}-1}\ dt=\sum_{n\geq 1}^{\mathcal{R}}\frac1{n^s}.
$$
The relation (22) follows immediately from (R15).

\medskip

 \textbf{Proof of  lemma 3}. 
 
  We have
 $$
 \Big\vert \frac1{2+it}\big(\frac{it}{2+it}\big)^m-\frac1{2-it}\big(\frac{-it}{2-it}\big)^m\Big\vert\leq 2\Big\vert \frac1{2+it}\big(\frac{it}{2+it}\big)^m\Big\vert=\frac2{\sqrt{t^2+4}}\big(\frac t{\sqrt{t^2+4}}\big)^m
$$
Using  lemma 1,  we see that  the terms of the  series  (25) are, for $m\geq 1$,  dominated by the terms
$$
 \frac4{\sqrt{t^2+4}}e^{\pi\vert Im (s)\vert}\big(\frac t{\sqrt{t^2+4}}\big)^m\Big(\frac{\Gamma(\Re(s))}{m^{\Re(s)}}+\frac{\Gamma(1-\Re(s))}{m^{1-\Re(s)}}\Big) 
 $$
Consequently, the series (25) is convergent, and we have
 $$
 \sum_{m= 0}^{+\infty}\Big\vert\Big( \frac1{2+it}\big(\frac{it}{2+it}\big)^m-\frac1{2-it}\big(\frac{-it}{2-it}\big)^m\Big)\ P_m(s)\Big\vert\ 
 \leq M(t)
 $$
 with
 $$
 M(t)=\frac{4t}{4+t^2}+e^{\pi\vert Im (s)\vert}\frac4{\sqrt{t^2+4}} \sum_{m= 1}^{+\infty}\big(\frac t{\sqrt{t^2+4}}\big)^m\Big(\frac{\Gamma(\Re(s))}{m^{\Re(s)}}+\frac{\Gamma(1-\Re(s))}{m^{1-\Re(s)}}\Big) 
 $$
It remains to prove that
$$
\int_0^{+\infty} M(t)\frac1{e^{2\pi t}-1} \ dt <+\infty.
$$
This follows, for the first term in $M(t)$,  from
 $$
 \int_0^{+\infty} \frac{4t}{4+t^2}\frac1{e^{2\pi t}-1}\ dt<+\infty,
 $$
and  for the others terms from  the fact that, for $0<\alpha<1$, we have
$$
\int_0^{+\infty}\frac 1{\sqrt{t^2+4}}\sum_{m= 1}^{+\infty}\big(\frac t{\sqrt{t^2+4}}\big)^m\frac1{m^{\alpha} } \frac1{e^{2\pi t}-1}\ dt<+\infty,
$$
as a consequence of the following inequalities
$$
\frac 1{\sqrt{t^2+4}}\sum_{m= 1}^{+\infty}\big(\frac t{\sqrt{t^2+4}}\big)^m\frac1{m^{\alpha} }
\leq \frac 1{\sqrt{t^2+4}}\sum_{m= 1}^{+\infty}\big(\frac t{\sqrt{t^2+4}}\big)^m
\leq\frac t{{t^2+4}}\Big(\frac1{ 1-\frac t{\sqrt{t^2+4}}}\Big)\leq \frac{t^2}4.
$$

\end{proof}

 \subsection{Evaluation of  the sum  $\sum_{n\geq 1}^{\mathcal{R}}\frac1{n+1}\Big(\frac{n-1}{n+1}\Big)^m$}

\begin{lemma}
For any integer $m\geq 0$, we have
$$
\sum_{n\geq 1}^{\mathcal{R}}\frac1{n+1}\Big(\frac{n-1}{n+1}\Big)^m
=x_m+y_m
 $$
 with  $ x_0=\gamma-1,\ \  y_0=\log(2),$ and for $m\geq 1$
 
 \begin{equation} 
x_m=\gamma-1+\sum_{k=1}^m \binom{m}{k}(-2)^k\Big(\zeta(k+1)-1-\frac1k\Big) ,
\end{equation}
\begin{equation} 
y_m=\int_0^1\Big(\frac{x-1}{x+1}\Big)^m\frac1{x+1}\ \ dx=\log(2)+\sum_{k=1}^m \binom{m}{k}(-2)^k\frac1k\Big(1-\frac1{2^k}\Big).
\end{equation}
 
 \end{lemma}
 
 \medskip

 \begin{proof}

We apply  the shift formula (R12)  to the  function $f:x\mapsto \frac1{x}\big(1-\frac{2}{x}\big)^m$, then we get 
 $$
\sum_{n\geq 1}^{\mathcal{R}}\frac1{n+1}\Big(\frac{n-1}{n+1}\Big)^m=\sum_{n\geq 1}^{\mathcal{R}}\frac1{n}\Big(1-\frac{2}{n}\Big)^m-(-1)^m+\int_0^1\Big(\frac{x-1}{x+1}\Big)^m\frac1{x+1}\ dx.$$
For any $m$ integer $\geq 0$, let  
\begin{equation*} 
x_m=\sum_{n\geq 1}^{\mathcal{R}}\frac1{n}\Big(1-\frac{2}{n}\Big)^m-(-1)^m\ 
\ \text{ 
and }\  \ 
y_m=\int_0^1\Big(\frac{x-1}{x+1}\Big)^m\frac1{x+1}\ dx.
\end{equation*}
We have $ x_0=\gamma-1,\ \  y_0=\log(2).$  
By the binomial expansion
we get 
for  $m\geq 1$
\begin{eqnarray*}
x_m
&=&\gamma+\sum_{k=1}^m \binom{m}{k}(-2)^k\sum_{n\geq 1}^{\mathcal{R}}\frac1{n^k} -\sum_{k=1}^m \binom{m}{k}(-2)^k-1.
\end{eqnarray*}
and 
$$
y_m=\int_1^2\frac1{x}\big(1-\frac{2}{x}\big)^m\ \ dx=\sum_{k=0}^m \binom{m}{k}(-2)^k\int_1^2\frac1{x^{k+1}}\ dx.
$$
The conclusion follows from (R15) which gives for $k\geq 1$
$$\sum_{n\geq 1}^{\mathcal{R}}\frac1{n^k} =\zeta(k+1)-\frac1k.$$
\end{proof}

 \subsection{ The polynomial expansion}

\begin{theorem}
{\it For  $0<\Re(s)<1$, we have
\begin{equation}
\zeta(s)=\frac{1}{s-1}
+\sum_{m= 0}^{+\infty}z_m\, P_m(s),
\end{equation}
with
$$
z_m=\gamma+\log(2)+\sum_{k=1}^m \binom{m}{k}(-2)^k\zeta(k+1)+(-1)^{m+1}
+\sum_{k=1}^m \frac{1}k.
$$
}
\end{theorem}

\begin{proof}

This follows from  the expansion (22) of  theorem 4.
Since by lemma 4, we have 
\begin{eqnarray*}
z_m=x_m+y_m
=\gamma+\log(2)+\sum_{k=1}^m \binom{m}{k}(-2)^k\zeta(k+1)+(-1)^{m+1}
+\sum_{k=1}^m \binom{m}{k}\frac{(-1)^{k-1}}k,
\end{eqnarray*}
 and we can use 
the combinatorial identity
$$
\sum_{k=1}^m \binom{m}{k}\frac{(-1)^{k-1}}k=\sum_{k=1}^m \frac{1}k.$$
\end{proof}

 The  previous  theorem will enable us to get an expansion of zeta in the strip   $\{0<\Re(s)<1\}$. Indeed, the polar part  $1/(s-1)$ of (28) has  the simple  expansion given by the following lemma.
 
\begin{lemma}
For  $0<\Re(s)<1$ we have
\begin{equation}
 \frac1{s-1}=-\sum_{m= 0}^{+\infty}y_m\, P_m(s).
\end{equation}
\end{lemma}

\begin{proof}

For $0<\Re(s)<1$, we   write
\begin{equation}
\sum_{m= 0}^{+\infty}y_mP_m(s)=\sum_{m= 0}^{+\infty}\Big(\int_0^1\Big(\frac{x-1}{x+1}\Big)^m\frac1{x+1}\ \ dx\Big)
P_m(s)=\int_0^1\sum_{m= 0}^{+\infty}\frac1{x+1}\Big(\frac{x-1}{x+1}\Big)^mP_m(s)\ dx.
\end{equation}
This interchanging of signs $\sum$ and $\int$ is justified by  
$$
\sum_{m= 0}^{+\infty}\int_0^1\vert\frac{x-1}{x+1}\vert ^m\  
\vert P_m(s)\vert \frac1{x+1} \ dx\leq \sum_{m= 0}^{+\infty}\int_0^1\frac{1}{(x+1)^{m+1}}\ dx\ 
\vert P_m(s)\vert  
\leq 
\sum_{m= 0}^{+\infty}\frac{1}{m}
\vert P_m(s)\vert, 
$$
and we immediately check, with  lemma 1, that  the latter series is  convergent.

\noindent  Since we have by (7)
$$
\sum_{m= 0}^{+\infty}\frac1{x+1}\Big(\frac{x-1}{x+1}\Big)^mP_m(s)=x^{-s}
$$
then relation (29) is a consequence of (30).

\end{proof}

\medskip

\begin{theorem}

For $ 0<\Re(s)<1$,  we have 
\begin{equation}
\zeta(s)=\sum_{m= 0}^{+\infty}x_m\, P_m(s)
\end{equation}
with $x_0=\gamma-1,$ and for $m\geq 1$
\begin{equation}
x_m=\gamma+\sum_{k=1}^m \binom{m}{k}(-2)^k\zeta(k+1) +r_m
\end{equation}
with $r_m$ the rational number
 $$
 r_m=(-1)^{m+1}+\sum_{k=1}^m \frac{1+(-1)^{k-1}}{k}.$$
\end{theorem}

\begin{proof}
This is an immediate  consequence of the preceding results and the combinatorial identity
$$
\sum_{k=1}^m \binom{m}{k}\frac{(-2)^{k-1}}k=\sum_{k=1}^m \frac{1+(-1)^{k-1}}{k}.
$$

\end{proof}

For the first values of  $m$ we get 
\begin{eqnarray*}
x_0&=&\gamma-1,\\
x_1&=& \gamma-2\zeta(2)+3,\\
x_2&=& \gamma-4\zeta(2)+4\zeta(3)+1,\\
x_3&=& \gamma-6\zeta(2)+12\zeta(3)-8\zeta(4)+11/3,\\
x_4&=& \gamma-8\zeta(2)+24\zeta(3)-32\zeta(4)+16\zeta(5)+5/3.\\
 \end{eqnarray*}
 
\begin{remark} 
Since by (24)  we have, 
$
\overline{Q_m(t)}=(-1)^mQ_m(t),
$
for every $m\geq 0$, then we have for $t\in \mathbb R$

\begin{eqnarray*}
\Re\zeta(\frac{1}2+it)&=&2\sum_{m=0}^{+\infty} x_{2m}Q_{2m}(t),\\
\Im\zeta(\frac{1}2+it)&=&-2i\sum_{m=0}^{+\infty} x_{2m+1}Q_{2m+1}(t).
\end{eqnarray*}
\end{remark}

\begin{remark}
Using the generating function (6), we verify that $P_{2m}(\frac12)=2\binom{2m}m\frac{1}{2^{2m}}$ and $P_{2m+1}(\frac12)=0$.
Thus the relation $(1-\sqrt 2)\zeta(\frac12)= \eta(\frac12)$ gives 
$$
\sum_{m= 0}^{+\infty}\binom{2m}m\frac1{2^{2m}}
{((1-\sqrt 2)x_{2m}-a_{2m})}=0.
$$
which is a sort of infinite relation between the constants  $\gamma$, $\log(2)$, $\sqrt 2$ and the zeta values.
\end{remark}

\subsection{An orthonormal expansion}

 Let $X(0)=\gamma-1$ and  for $k\geq 1$ let 
 $$
 X(k)=\zeta(k+1)-1-\frac1k .$$
 Then formula  (26) can be written as a binomial transform 
 \begin{equation}
x_m=\sum_{k=0}^m \binom{m}{k}(-1)^k2^kX(k).
\end{equation} 
Let the function $f_0$ defined for $t>0$ by 
$$
f_0(t)=\frac1{e^t-1}-\frac1t,
$$
it is well-known that  $\mathcal{M}f_0(s)=\Gamma(s)\zeta(s)$ for $0<\Re(s)<1$.
We easily verify that  for $k\geq 0$ we have 
$$
X(k)=\int_0^{+\infty}e^{-t} \frac{t^k}{k!}\  f_0(t)\ dt.  
$$
Then by (33) and the definition of Laguerre polynomials, we obtain 
\begin{equation*}
x_m= \int_0^{+\infty}e^{-t} L_{m}(2t)f_0(t)\ dt
=\frac1{\sqrt2}\int_0^{+\infty}\varphi_m(t)f_0(t) \ dt=\frac1{\sqrt2}(\varphi_m,f_0). 
\end{equation*}
Since $f_0$ is in the space $L^2(]0,+\infty[,dx)$, this means that we have in this space
 the orthonormal expansion
 \begin{equation}
f_0=\sqrt{2}\, \sum_{m=0}^{+\infty} x_m\, \varphi_m.
\end{equation}

Just as in Theorem 3 on the eta function,   we get  an  orthonormal expansion for zeta if we restrict $s$ to the line $\{\Re(s)=\frac12\}$.

\medskip

\begin{theorem}
For any integer $m\geq 0$, we have 
\begin{equation}
x_m=\frac1{2}\int_{-\infty }^{+\infty }\overline{Q_m(t)}\ \zeta(\frac12+it)\frac{dt}{\cosh(\pi  t)}.
\end{equation}

In the space $L^2\big(\mathbb R,\frac{dt}{\cosh(\pi t)}\big)$, the function 
$\zeta_{1/2}:t\mapsto \zeta(\frac{1}2+it)$ has the orthonormal expansion
$$
\zeta_{1/2}=2\sum_{m=0}^{+\infty} x_m\, Q_m.
$$

\end{theorem}
  
  \medskip
  
\begin{proof}
By  Mellin-Parseval  formula (cf. [T]), we have
$$
x_m=\frac1{\sqrt2}\int_{0}^{+\infty }{\varphi_m(t)}f_0(t)\ dt=\frac{1}{2\pi\sqrt2 }\int_{-\infty }^{+\infty }
\overline{\mathcal{M}(\varphi_m)(\frac{1}2+it)}\ \mathcal{M}(f_0)(\frac{1}2+it)\ dt.
$$
 Thus we get
\begin{eqnarray*}
x_m
&=&\frac{1}{2}\int_{-\infty }^{+\infty }%
\overline{\frac12P_m(\frac12+it)}\ \zeta(\frac{1}2+it)\frac{dt}{\cosh(\pi  t)}.
\end{eqnarray*}
 The second assertion follows from the fact that the function $t\mapsto \zeta_{1/2}$ is in $L^2\big(\mathbb R,\frac{dt}{\cosh(\pi t)}\big)$ since
$$
\int_{-\infty }^{+\infty }
\vert \zeta(\frac{1}2+it)\vert ^2\frac{dt}{\cosh(\pi  t)}
<+\infty.
$$

\end{proof}

\medskip

\begin{remark}

By (34) and (35) we deduce that 
\begin{equation}
\int_{-\infty }^{+\infty }\Big\vert \zeta(\frac12+it)\Big\vert^2\frac{dt}{\cosh(\pi  t)}=4\sum_{m=0}^{+\infty} x_m^2=2\int_0^{+\infty}\Big(\frac1{e^t-1}-\frac1t\Big)^2dt=2\log(2\pi)-2\gamma-1.
\end{equation}

\end{remark}

\medskip

\begin{corollary}
We have    
  \begin{equation} 
\gamma=1+\frac1{2}\int_{-\infty }^{+\infty }\zeta(\frac12+it)\frac{dt}{\cosh(\pi  t)},
  \end{equation}
and for $m\geq 1$
  \begin{equation}
\zeta(m+1)=1+\frac1m+\frac12\int_{-\infty }^{+\infty }\zeta(\frac12+it)\frac{(\frac12-it)_m}{m!}\frac{dt}{\cosh(\pi  t)}.
  \end{equation}

\end{corollary}

\medskip

\begin{proof}

With (33) and
 inversion of the binomial transform, we get
  \begin{equation*}
\sum_{k=0}^m \binom{m}{k}(-1)^kx_k=
2^mX(m).
\end{equation*} 
As in corollary 2 we obtain by (35) the integral expression
 
\begin{eqnarray*}
X(m)
&=&\frac1{2^{m+1}}\int_{-\infty }^{+\infty }%
\sum_{k=0}^m \binom{m}{k}(-1)^k\overline{Q_k(t)}\ \zeta(\frac{1}2+it)\frac{dt}{\cosh(\pi  t)}\\
\end{eqnarray*}
and by (9) this is 
$$
X(m)=\frac12\int_{-\infty }^{+\infty }\frac{(\frac12-it)_m}{m!}\zeta(\frac12+it)\frac{dt}{\cosh(\pi  t)}.
$$

\end{proof}

\newpage 
 
\section{Appendix : Ramanujan's summation}
Below is a brief overview of the Ramanujan's  summation  method (cf. [Can]).

\subsection{Ramanujan's constant of a series}

Let a function $f $ in $C^{1 }(]0,+\infty \lbrack )$, then for every integer $n\geq1$, we have the first order  Euler-MacLaurin formula 
\begin{equation} \label{R1} 
{\sum_{k=1}^{n}f(k)=\int_{1}^{n}f(x)\, d x+\frac{1}{2}\big(f(n)+f(1)\big)+%
\int_{1}^{n}b(x)f^{\prime }(x)\ d x},\tag{R1}
\end{equation}
where $b$ is the periodic function defined by
$$
 b(x)=x-[x]-\frac12 \ \text{ where } \ [x] \ \text{ denote the integer part of }\ x.
 $$ 
 Let's assume that the integral  $\int_{1}^{+\infty}b(x)f^{\prime }(x)\ d x$ is convergent. Then the  previous formula gives  
\begin{align}  \label{R2}
\begin{split}
\sum_{k=1}^{n}f(k)&=\int_{1}^{n}f(x)\, d x+C(f)+\frac{1}{2}f(n)-\int_{n}^{+\infty}b(x)f^{\prime }(x)\ d x,\\
\text{ with  }\ & C(f)=\frac{1}{2}f(1)+
\int_{1}^{+\infty}b(x)f^{\prime }(x)\ d x.
\end{split}\tag{R2}
\end{align}
The constant $C(f)$ is called \textit{the Ramanujan constant of the series $\sum_{n\geq 1}f(n)$} (cf. [H1]). 

\medskip

\noindent Clearly $C(f)$ depends linearly of $f$, 
and if the series $\sum_{n\geq 1}f(n)$ is  convergent, then by (R2) we have
$$
C(f)=\sum_{k=1}^{+\infty}f(k)-\int_{1}^{+\infty}f(x)\ d x.
$$
We observe that $C(f)$ is  also well defined for some divergent series :
for example, we have  immediately $C(1)=1/2,$  and applying (R2) to  the function $f:x\mapsto   1/x$, 
we have  for any integer $n\geq1$
\begin{equation*}
\sum_{k=1}^{n}\frac1k=\log n+C(f)+\frac{1}{2n}+%
\int_{n}^{+\infty}b(x)\frac1{x^2}\ d x,
\end{equation*}
and taking the limit as $n\to +\infty$, we see that   Ramanujan's constant of the divergent series  $\sum_{n\geq 1}  \frac1n$ is the \textit{Euler constant} :
$$
\gamma=\lim_{n\to +\infty}\Big(\sum_{k=1}^{n}\frac1k-\log n\Big).$$

\subsection{Plana's formula}
The following summation formula was discovered by  Plana in 1820  (also by Abel in 1823) (cf. [Lin]).
\medskip

\begin{theorem}

  Let  $f$ be an analytic function in the half-plane  $\{ \Re(z)>0\}$ such that there is a constant    ${a}<2\pi$ and, for any interval $J=[c,d]\subset]0,+\infty[$,  a constant $M_J>0$ with
\begin{equation*} 
\left| f(x+it)\right| \leq M_J\ e^{a\vert t\vert },\  \text{ for  } x\in J\ \text{ and } t\in \mathbb{R}.
\end{equation*}
Then  we can write Plana's formula : for any integer  $n\geq 1$ we have 
\begin{align*}
\sum_{k=1}^{n}f(k) &=\frac{f(1)+f(n)}{2}+\!\int_{1}^{n}f(u)\ d u\nonumber\\
&+i\!\int_{0}^{+
\infty }\frac{f(1+it)-f(1-it)}{e^{2\pi t}-1}\,\ d t 
-i\!\int_{0}^{+\infty }\frac{f(n+it)-f(n-it)}{e^{2\pi t}-1}\,\ d t\,.
\end{align*}
\end{theorem}

\medskip

If we now compare  Euler-MacLaurin formula (R1) 
and  Plana's formula,  we get, for $n\geq 1$,
\begin{equation*} \int_1^nb(x){f'(x)}\ dx= i\!\int_{0}^{+
\infty }\frac{f(1+it)-f(1-it)}{e^{2\pi t}-1}\ d t 
-i\int_{0}^{+\infty }\frac{f(n+it)-f(n-it)}{e^{2\pi t}-1}\,\ d t\,.
\end{equation*}
Then if the integral $\int_{1}^{+\infty}b(x)f^{\prime }(x)\, d x $ is convergent and if  
$$ 
\lim_{ n\to+\infty}\int_{0}^{+\infty }\frac{f(n+it)-f(n-it)}{e^{2\pi t}-1}\ d t= 0, $$
we have 
$$ 
\int_1^{+\infty}b(t){f'(t)}\ d t= i\int_{0}^{+
\infty }\frac{f(1+it)-f(1-it)}{e^{2\pi t}-1}\,\ d t .
$$
Thus, in  this case  the Ramanujan constant $C(f)$, previously  defined in (R2),  is  also
\begin{equation} \label{R3}
C(f)=\frac{1}{2}f(1)+i\int_{0}^{+
\infty }\frac{{f(1+it)-f(1-it)}}{e^{2\pi t}-1}\,\ d t .\tag{R3}
\end{equation}

\begin{remark}
Let  $f$ and $g$ two functions  verifying the assumptions of theorem 8 and such that 
  $f(n)=g(n) $ for every integer  $n\geq 1$. Then  the Ramanujan constants  $C(f)$ and $C(g)$,  defined by 
 (R3), need not  be  the same.
For example, if $f\equiv0$ and  $g : z\mapsto \sin(\pi z)$, then $C(f) =0$, but 
 \begin{align*}
C(g)=i\int_{0}^{+\infty }\frac{\sin \big(\pi (1+it)\big)-\sin \big(\pi (1-it)\big)}{e^{2\pi t}-1}\,\ d t =\int_{0}^{+\infty }
\frac{e^{\pi t}-e^{- \pi t}}{e^{2\pi t}-1}\,\ d t
\;\;=\; \int_{0}^{+\infty }
e^{-\pi t}\ d t= \frac1{\pi}.
\end{align*}
To avoid  this problem, we will consider a space of analytic functions $f$ such that the Ramanujan constant $C(f)$,  defined by   (R3), only depends 
 on  the values   $f(n) $ for all integers  ${n\geq 1}$.

\end{remark}

\subsection{The space $\mathcal{E}$}

\medskip

\begin{definition}

Let   ${\mathcal E}$   the space of  analytic functions $f$ in the half plane $\{\Re(z)>0\}$ such that there are some constants  $\emph{a}<\pi$ and $\emph{b}>0$ with :

for every  ${\varepsilon}>0$,  there is a constant $C_{\varepsilon}$ such that  
\begin{equation} \label{R4} 
\left| f(z)\right| \;\;\leq\;\; C_{\varepsilon}\ e^{b\, \Re(z)} \ e^{a\, |\Im (z)|},\ \text{ if } \ \Re(z)>{\varepsilon}.\tag{R4}
\end{equation}

\end{definition}

\medskip

\begin{example} Let $f_s\;:\;z\mapsto e^{sz}$, with $s=\sigma+i\tau$, $\sigma\in \mathbb{R}$ and $ \tau\in \mathbb{R} $ .   For $\Re(z)>0$, we have
$$
\vert f_s(z) \vert \;\;=\,\;e^{\sigma \Re(z)-\tau \Im (z)}.
$$
Thus  $f_s \in \mathcal E$ if and only if $ -\pi< \Im(s)<\pi$.

\end{example}

\medskip

An important theorem for the rest of this section is the following.

\medskip

\begin{theorem}\textbf{Carlson's theorem} (cf.  [H2])

If $f\in \mathcal E$ is such that  $f(n)=0$ for all $n\;=\;1,\,2,\,\ldots$,  then  $f\equiv 0$.

Thus a function  $f\in \mathcal  E$ is only determined by the values  $f(1),f(2),f(3),\dots$.

\end{theorem}

\medskip
\begin{remark}

The function $g\;:\;z\mapsto \sin(\pi z)$ is such that   $g(n) =0 $ for every integer ${n\geq 1}$,  but $g\notin \mathcal E$ since
for $x\in \mathbb{R}$ and $t\in \mathbb{R}$ 
$$
\vert g(x+it)\vert\sim \frac12e^{\pi \vert t \vert} \text{ when } t\to \infty.
$$ 
\end{remark}

\subsection{ Ramanujan's method of summation}

Let  $f$ and $g$ two functions  verifying the hypothesis of theorem 8.
By  Carlson's theorem we know that if  $f\in \mathcal E$ and $g\in \mathcal E$ verify ${f(n)=g(n)}$ for  $n=1,2,3,\dots$, then for  ${x>0}$ and $t\in \mathbb{R}$, we have 
$$ {f(x+it)-f(x-it)}={g(x+it)-g(x-it)},$$ 
thus, by (R3), we have
 $C(f)=C(g).$
 This leads to the following definition.   
\medskip

\begin{definition}
For  $f\in \mathcal E$,  we define  \textit{the Ramanujan's summation of  the series} $\sum_{n\geq1}f(n)$ by 
\begin{equation} \label{R5} 
\sum_{n\geq1}^{\mathcal{R}}f(n)=C(f)=\frac{f(1)}{2}
+ i\!\int_{0}^{+\infty }\frac{f(1+it)-f(1-it)}{e^{2\pi t}-1}\,\ d t\,.\tag{R5}
\end{equation}

\end{definition}

\medskip
It is clear that  $\sum_{n\geq1}^{\mathcal{R}}f(n)
$ depends linearly  on $f$.

\medskip

\begin{remark}
 
Let  $f\in \mathcal E$ then then $f$ satisfies the hypotheses of theorem 8, and by  Plana's formula, we have for every integer $n\geq 1$

\begin{equation} \label{R6}
\sum_{k=1}^{n}f(k)=\!\int_{1}^{n}f(u)\ d u+\frac{f(n)}{2}+\sum_{n\geq1}^{\mathcal{R}}f(n)-i\int_{0}^{+\infty }\frac{f(n+it)-f(n-it)}{e^{2\pi t}-1}\,\ d t\,.\tag{R6}
\end{equation}
If 
 $\lim_{n\to +\infty}f(n+it)=0$ for every  $\ t\in \mathbb{R}$, 
then, by dominated convergence, we get  
$$
\lim_{n\to +\infty}\int_{0}^{+\infty }\frac{f(n+it)-f(n-it)}{e^{2\pi t}-1}\,\ d t\,=0,
$$
and if in addition,  the integral 
 $\int_{1}^{+\infty }f(u)\ d u $ is convergent,
 then the series $ \sum_{n\geq1}f(n)$ is convergent. 
In this case  (R6) gives the relation
\begin{equation*} 
  \sum_{n\geq1}^{\mathcal{R}}f(n)\;\;=\sum_{n=1}^{+\infty} f(n)-\!\!\int_{1}^{+\infty}\!\!f(x)\ d x.
\end{equation*}
\end{remark}

\medskip
\begin{example}

Let $f:x\mapsto \frac1{x^s}$, then for $\Re(s)>1$, we get  
\begin{equation*} 
  \sum_{n\geq1}^{\mathcal{R}}\frac1{n^s}\;\;=\sum_{n=1}^{+\infty} \frac1{n^s}-\!\!\int_{1}^{+\infty}\!\!\frac1{x^s}\ d x,
\end{equation*}
thus, for $\Re(s)>1$, we have 
\begin{equation}  \label{R7}
  \sum_{n\geq1}^{\mathcal{R}}\frac1{n^s}\;\;=\zeta(s)-\frac1{s-1}.\tag{R7}
  \end{equation}
  \end{example}

\subsection{Interpolation of  partial sums}
\medskip

To evaluate Ramanujan's summation of a divergent series $\sum_{n\geq1}f(n)$, a simpler  formula than (R5) can be obtained by means of  an interpolation function of the  partial sums.
In  chapter  VI of his  Notebooks,   Ramanujan uses, without  any definition, such a function $\varphi$
which satisfies
$${\varphi(x)-\varphi(x-1)=f(x)}\ \text{ with  } \varphi(0)=0.$$
The following theorem clarifies this point.
 \smallskip

\begin{theorem} Let  $f\in \mathcal E$.
There exists  a unique function  $\varphi_f\in \mathcal E$, such that   
\begin{equation}  \label{R8} 
\varphi_f(x+1)-\varphi_f(x)=f(x+1)\ \text{ for all }x>0
\text { and }
\varphi_f(1)=f(1). \tag{R8}
\end{equation}

\end{theorem}

\medskip

\begin{proof}

a)  \emph{Existence.}
 Since $f\in \mathcal E$ then the function $\varphi_f$, defined for  $\Re(z)>0$,  by 
\begin{equation} \label{R9} 
\varphi_f(z) =\int_{1}^{z}f(u)\ d u+\frac{f(z)}{2}+\sum_{n\geq1}^{\mathcal{R}}f(n)-i\int_{0}^{+\infty }\frac{f(z+it)-f(z-it)}{e^{2\pi t}-1}\ d t,\tag{R9}
\end{equation}
is also in 
$\mathcal E$.  
By  (R6) 
we have $\varphi_f(n)=\sum_{k=1}^{n}f(k)$ for every integer  $n\geq 1$, thus the function  $\varphi_f$ verifies  
$$
\varphi_f(n+1)-\varphi_f(n)\,=f(n+1) ,  \text{ for }  n= 1,2,3,\dots.$$
Since the functions  $z\mapsto \varphi_f(z+1)-\varphi_f(z)$ and $z\mapsto f(z+1)$ are in   $\mathcal E$, then by  Carlson's theorem, we get  
\begin{equation} \label{R10}
\varphi_f(z+1)-\varphi_f(z)=f(z+1),\quad\!\! \text{ for all }z \text{ such that }  \Re(z)>0.\tag{R10}
\end{equation}
 The condition $\varphi_f(1)=f(1)$ is simply a consequence of (R9) and definition  (R5).
 
 \medskip
 
 b) \emph{Uniqueness.}

If  $\varphi \in \mathcal E$ is such that   
$$ 
\varphi(x+1)-\varphi(x)=f(x+1)\  \text{ for }x>0\  \text{ and }
\varphi(1)=f(1),
$$
then, for every integer  $n\geq 1$,  we get  
$$ 
\varphi(n)-f(1)=\varphi(n)-\varphi(1)\;=\sum_{k=1}^{n-1}\big(\varphi(k+1)-\varphi(k)\big) =\sum_{k=1}^{n-1}f(k+1),
$$
thus   $\varphi(n)=\sum_{k=1}^{n} f(k)=\varphi_f(n)$ for every integer  $n\geq 1$. 
Since   $\varphi $ and $\varphi_f$ are in the space  $\mathcal E$, then by  Carlson's theorem we get  $\varphi =\varphi_f.$

\end{proof}

\begin{remark}

The preceding proof  asserts that if  $f\in \mathcal E$ then the function $\varphi_f$ given  by (R9) satisfies
$$
\varphi_f(n)=\sum_{k=1}^{n} f(k) \text{ for every integer } n\geq1.
$$
Thus this function $\varphi_f$  is the unique  function in  $ \mathcal E$ which   interpolates  the partial sums of the series  $\sum_{n\geq 1}f(n)$.
Note that, by (R10), we have for $\Re(z)>0$
 $$
\varphi_f(z)=\varphi_f(z+1)-f(z+1),
$$
and since the right-hand side defines  an analytic function  in the half-plane $\{\Re(z)>-1\}$, this relation gives an analytic continuation of  $\varphi_f$ in  this half-plane. Let's denote again by $\varphi_f$ this continuation, then the relation  ${\varphi_f(1)=f(1)}$ provides the  condition required by Ramanujan
$${\varphi_f(0)=\varphi_f(1)-f(1)=0.}$$
\end{remark}

\medskip

\begin{example}

Let  the function $f:z\mapsto   1/z$. We will see that   $\varphi_f$   is related to  {\it  digamma function} which is defined, for  $\Re(z)>0$, by 
$$
\Psi(z)=\frac{\Gamma'(z)}{\Gamma(z)}.
$$
The function $\Psi$ is analytic in   $\{\Re(z)>0\}$, 
and  (cf.  [G-R]) we have 
 for $\Re(z)>0$
\begin{equation*} 
\Psi(z)=-\frac1z-\gamma+\sum_{n=1}^{+\infty}\big(\frac1n-\frac1{n+z}\big)= 
-\frac1z-\gamma+\sum_{n=1}^{+\infty}\frac z{n(n+z)}.
\end{equation*}
With this expression it is easily verified that  $\Psi \in\mathcal E$ and 
that 
\begin{equation*}  
\Psi(z+2)\;-\;\Psi(z+1)=\frac1{z+1}.
\end{equation*}
Since $\Psi(1)=-\gamma$,
then $\Psi(2)+\gamma=1$ and by the previous theorem, we get  for $\Re  (z)>0$
$$
   \varphi_{f}(z)=\Psi(z+1)+\gamma.
$$

\end{example}

\medskip

\begin{theorem} 
Let $f\in \mathcal E$, and let $\varphi\in \mathcal E$ be a solution of 
\begin{eqnarray*}
\varphi(x+1)-\varphi(x)=f(x+1),\  \text{ for all }\ x>0,\ 
\text{ and }\ \varphi(1)=f(1).
\end{eqnarray*}
Then we have
$$\sum_{n\geq1}^{\mathcal{R}}f(n)\;\; =\int_0^1\varphi(x) \ dx.$$

\end{theorem}

\medskip

\begin{proof}

By  theorem 10, we have  $\varphi=\varphi_f$.
  Let $F(z)=\int_1^zf(u)\ du$, then integration of    (R9)  over  $[1,2]$ 
gives
\begin{equation}\label{R11}
\int_1^2\varphi_f(x) \ d x\;\;=\int_1^2F(u)\ d u\,+\,\frac12\int_1^2f(x)\ d x\,+\sum_{n\geq1}^{\mathcal{R}}f(n)-iD(f),\tag{R11}
\end{equation}

where  
$$ 
D(f)=\int_1^2\Big(\int_{0}^{+\infty }\frac{f(x+it)-f(x-it)}{e^{2\pi t}-1}\,\ d t\Big)\ d x.
$$
By Fubini's theorem   we have
$$
D(f)\;\; =\int_0^{+\infty}{\Big(\int_1^2{f(x+it)-f(x-it)}\ d x\Big) }\frac1{e^{2\pi t}-1} \ dt.
$$
Thus 
$$ 
D(f) = \int_0^{+\infty}\frac{F(2+it)-F(2-it)}{e^{2\pi t}-1}\ dt
-\!\int_0^{+\infty}\frac{F(1+it)-F(1-it)}{e^{2\pi t}-1}\ dt .
$$
\smallskip

\noindent If we apply  Plana's formula   to the function $F$, in the case  $n=2$, then we get  
\begin{eqnarray*}
i \int_0^{+\infty}\frac{F(2+it)-F(2-it)}{e^{2\pi t}-1}\ dt
-i\int_0^{+\infty}\frac{F(1+it)-F(1-it)}{e^{2\pi t}-1}\ dt=-\frac12F(1)-\frac12F(2)+\!\int_1^{2}F(t)\ dt,
\end{eqnarray*}
that is
$$
iD(f)=-\frac{F(2)}2+\int_1^2F(t)\ dt.
 $$
Then by (R11) we get
\begin{eqnarray*}
\int_1^2\varphi_f(x) \ d x&=&\frac12\int_1^2f(x)\ dx+\sum_{n\geq1}^{\mathcal{R}}f(n)
+\frac{F(2)}2\\
&=& \int_1^2f(x)\ dx+\sum_{n\geq1}^{\mathcal{R}}f(n),
\end{eqnarray*}
from which we deduce that 
$$\sum_{n\geq1}^{\mathcal{R}}f(n)
=\int_0^1\big(\varphi_f(x+1)-f(x+1)\big)\ dx =\int_0^1\varphi_f(x)\ dx.
$$
\end{proof}

\begin{corollary}
 
 Let $f\in \mathcal E$, then the function $x\mapsto f(x+1)$ belongs to   $\mathcal E$ and  we have \textit{the shift formula}
  \begin{equation} \label{R12}
  \sum_{n\geq1}^{\mathcal{R}}f(n+1)\;\; =\sum_{n\geq1}^{\mathcal{R}}f(n)-f(1)+\int_0^1f(x+1) \ dx.\tag{R12}  \end{equation}

\end{corollary}

\medskip

\begin{proof}

Let $g:x\mapsto f(x+1)$, then it is easily verified, by theorem 10, that we have for  $x>0$ 
$$
\varphi_g(x) =\varphi_f(x+1)-f(1).
$$
Thus we obtain
$$
\int_0^1\varphi_g(x) d x=\int_0^1\varphi_f(x+1)\ dx-f(1)=\int_0^1\varphi_f(x)\ dx+\int_0^1f(x+1) \ dx-f(1),
$$
and the conclusion follows from  theorem 11.
\end{proof}

\medskip

\begin{example}
 Let  $f:z\mapsto 1/z$, as we saw in Example 3, we have 
$
\varphi_{f}(x)=\Psi(x+1)+\gamma,
$
thus we get 
$$
  \sum_{n\geq1}^{\mathcal{R}}\frac1n=\int_0^1(\Psi(x+1)+\gamma )\ \ dx=
\int_0^1\frac{\Gamma'}{\Gamma}(x+1)\ dx+\gamma $$
and this gives again 
\begin{equation} \label{R13}
\sum_{n\geq1}^{\mathcal{R}}\frac1n=\gamma. \tag{R13}
\end{equation}
Observe that, by the shift formula (R12), we  have 
$$\sum_{n\geq1}^{\mathcal{R}}\frac1{n+1} =\gamma-1+\log(2).$$
\end{example}

\medskip

\begin{example}
Let us evaluate $ \sum_{n\geq 1}^{\mathcal{R}}n^k$ for  any $k$ integer $\geq 1$. Recall that  the  \textit{Bernoulli polynomials} $B_{k}(x)$ are defined by 
$B_0(x)=1$ and, for  $k\geq 1, $
$$
B'_k(x)=kB_{k-1}(x) \ \text{ with }\ \int_0^1B_k(x)\ dx=0,
$$
and the \textit{Bernoulli  numbers} are the numbers $B_k=B_k(0)$. For $k\geq 1$, we have  $B_{k+1}(1)=B_{k+1}$.
These  polynomials verify (cf. [G-R])
$$
B_{k+1}\big((x+1)+1\big)-B_{k+1}(x+1)=(k+1)(x+1)^{k}.
$$
Let  the function  $f:x\mapsto x^k$  with $k$ integer $\geq 1$. Then 
  by theorem 10, we obtain
\begin{equation*}
\varphi_{f}(x)=\frac{B_{k+1}(x+1)}{k+1}-\frac{B_{k+1}}{k+1}.
 \end{equation*}
Since
$$
\int_0^1\frac{B_{k+1}(x+1)}{k+1}\ dx =\int_0^1\frac{B_{k+1}(x)}{k+1}\ dx\,+\int_0^1x^k
\ dx=\frac{1}{k+1},
$$
then theorem 11 gives  
\begin{equation} \label{R14}
\sum_{n\geq 1}^{\mathcal{R}}n^k=\frac{1-B_{k+1}}{k+1}.  \tag{R14}
 \end{equation}

 \end{example}

 \begin{remark}
 Relations  (R13) and (R14) can be written in terms of the zeta function :
\begin{eqnarray*}
 \sum_{n\geq 1}^{\mathcal{R}}\frac1n&=&\lim_{s\to 1}\Big(\zeta(s)-\frac1{s-1}\Big),\\
  \sum_{n\geq 1}^{\mathcal{R}}n^k&=&\zeta(-k)-\frac1{-k-1} \text{  for every integer }  k\geq 1. 
  \end{eqnarray*}
Note also that $$\sum_{n\geq 1}^{\mathcal{R}}n^0=\frac12=\zeta(0)+1.$$
  \medskip
These formulae are extensions of  (R7) and  are special cases  the following general result.

\end{remark}

 \begin{theorem}
 The function 
\begin{equation*}
s\mapsto \sum_{n\geq1}^{\mathcal{R}}\frac1{n^s}\;\;
\end{equation*}
is analytic on $ \mathbb{C}$ and we have for every 
 $s\in\mathbb{C}\setminus  \{1\}$

\begin{equation} \label{R15}
\zeta(s)= \sum_{n\geq 1}^{\mathcal{R}}\frac1{n^s}+\frac1{s-1}. \tag{R15} 
\end{equation}

 \end{theorem}

 \medskip
 
\begin{proof}
 
Let $s=\sigma +i\tau$ with $\sigma\in \mathbb{R} ,\  \tau\in \mathbb{R}$,  and  $z=x+it$ with  $x>\varepsilon>0$,  $t\in \mathbb{R}$.
We have 
$$ 
\Big\vert \frac1{z^s}\Big\vert=\Big\vert \frac1{(x+it)^{s}}\Big\vert =\Big\vert e^{-s\log(x+it)}\Big\vert=e^{-{\sigma} \log\sqrt{x^2+t^2} +\tau \arctan(\frac tx)}\leq e^{\vert \tau\vert \pi/2 }(x^2+t^2)^{-\sigma/2}
$$
Thus  for every   $s\in \mathbb{C}$, the function $z\mapsto   1/{z^s}$ is in the space $\mathcal E$ and the sum  $ \sum_{n\geq 1}^{\mathcal{R}}\frac1{n^s}$ is well defined. By definition (R5), we have  
$$
 \sum_{n\geq 1}^{\mathcal{R}}\frac1{n^s}= \frac{1}{2}
+i\int_{0}^{+\infty }\Big({\frac1{(1+it)^s}-\frac1{(1-it)^s}}\big)\frac1{e^{2\pi t}-1}\,\ dt.
$$

This last integral depends analytically of $s\in \mathbb{C}$ so the function  $s\mapsto \sum_{n\geq1}^{\mathcal{R}}\frac1{n^s}$ is analytic on  $\mathbb{C}$. 
Since  the function $s\mapsto \zeta(s)-\frac1{s-1}$ is analytic on $\mathbb{C}\setminus  \{1\}$ then, by analytic continuation of (R7), we get 
$$ 
 \sum_{n\geq 1}^{\mathcal{R}}\frac1{n^s}=\zeta(s)-\frac1{s-1}
 \text{ for any } s\in \mathbb{C}\setminus  \{1\}.
 $$

\end{proof}

\end{document}